\documentclass[11pt,a4j,oneside]{article}
\usepackage{amscd,amsmath,amssymb,amsfonts,mathrsfs,bm}
\usepackage{amsthm}
\usepackage[abbrev]{amsrefs}
\usepackage[all]{xy}
\usepackage{color}
\usepackage{graphicx}
\usepackage{latexsym}

\numberwithin{equation}{section}

\theoremstyle{definition}
\newtheorem{thm}{Theorem}[section]
\newtheorem{prop}[thm]{Proposition}
\newtheorem{lem}[thm]{Lemma}

\newtheorem{dfn}[thm]{Definition}

\newtheorem{rem}[thm]{Remark}

\begin{document}

\title{\huge Jones Wenzl projectors in Verma modules}
\author{Ryoga Matsumoto}
\date{}
\maketitle

\begin{abstract}
We construct special idempotents in $\mathrm{End}_{U_q(\mathfrak{sl}_2)}(M(\mu_1)\otimes\cdots \otimes M(\mu_n))$ like the Jones Wenzl projector where $M(\mu_i)$ is Verma module whose highest weight is $\mu_i$ and is complex number except non-negative integer.
\end{abstract}

\tableofcontents

\section{Introduction}
Temperley Lieb algebra is introduced in \cites{TemperleyLieb} to solve the problems of statistical physics. Furthermore, Wenzl introduced the special idempotents in Temperley Lieb algebra called Jones Wenzl idempotent in \cites{Wenzl}. The Jones Wenzl idempotent plays an important role in topology. Murakami and Murakami constructed knot invariants using the Jones Wenzl idempotent in \cite{MurakamiMurakami}. The knot invariants extend the Jones polynomial obtained in \cites{Jones}. It is known that Temperley Lieb algebra is isomorphic to the endomorphism algebra of tensor products of $2$-dimensional irreducible representation over $U_q(\mathfrak{sl}_2)$ in \cites{AndersenLehrerZhang}. Then it is important to consider the idempotents in endomorphism algebras of tensor products of $U_q(\mathfrak{sl}_2)$ representations. In \cites{LacabanneTubbenhauerVaz}, the structure of the endomorphism algebra of tensor products of Verma modules over $U_q(\mathfrak{sl}_2)$ is determined using Howe duality-like method. However, the special idempotent of the endomorphism algebra like the Jones Wenzl idempotent has not been constructed. In this article, we construct special idempotents in $\mathrm{End}_{U_q(\mathfrak{sl}_2)}(M(\mu_1)\otimes\cdots \otimes M(\mu_n))$ like Jones Wenzl projector where $\mu_i$ is complex number without non-negative integer. More precisely, we obtain the following.
\begin{thm}
There exists an element $P_{\mu_1, \ldots, \mu_n} \in \mathrm{End}_{U_q(\mathfrak{sl}_2)}(M(\mu_1)\otimes\cdots \otimes M(\mu_n))$ as follows.
\begin{align*}
P_{\mu_1, \ldots, \mu_n}^2&=P_{\mu_1, \ldots, \mu_n} \\
\end{align*}
\end{thm}

\subsection*{Acknowledgements}
We would like to thank Yuji Terashima for valuable discussions. This work was supported by JST, the establishment of university fellowships towards the creation of science technology innovation, Grant Number JPMJFS2102.

\section{Representation of $U_q(\mathfrak{sl}_2)$}
In this section, we define the quantum algebra $U_q(\mathfrak{sl}_2)$ and the representations. Hereinafter, suppose that $q$ is generic.
\begin{dfn}
The quantum group $U_q(\mathfrak{sl}_2)$ of $\mathfrak{sl}_2 $ is the algebra on $\mathbb{C}(q)$ generated by the elements $K,K^{-1},E,F$ that satisfy the following relations.
\begin{align*}
KK^{-1}=1=K^{-1}K,\quad KE=q^2EK \quad
KF=q^{-2}FK,\quad EF-FE=\frac{K-K^{-1}}{q-q^{-1}}
\end{align*}
The coproduct $\Delta: U_q(\mathfrak{sl}_2)\rightarrow U_q(\mathfrak{sl}_2)\otimes U_q(\mathfrak{sl}_2)$ is defined below on the algebra.
\begin{align*}
\Delta(K^{\pm1}):=K^{\pm1}\otimes K^{\pm1},\quad \Delta(F):=F\otimes 1+K^{-1} \otimes F, \quad \Delta(E):=E\otimes K+1 \otimes E \\
\end{align*}
\end{dfn}

\begin{dfn}
Let $\mu$ be a complex number. Then,
\[
[\mu]:=\frac{q^\mu-q^{-\mu}}{q-q^{-1}}
\]
\end{dfn}

\begin{dfn}
$k+1$-dimensional irreducible representation $V_k$ over $\mathbb{C}(q)$ has a basis $\{v_0, v_1, \cdots , v_k \}$ called an induced basis and satisfies the following relations for the generators of $U_q(\mathfrak{sl}_2)$.
\begin{align*}
K\cdot v_i:= q^{k-2i}v_i \\
E\cdot v_i:= [i]v_{i-1} \\
F\cdot v_i:= [k-i]v_{i+1}
\end{align*}
where $v_{-1}=v_{k+1}=0$.
\end{dfn}

\begin{dfn}
Let $k$, $l$ be non-negative integers such that $l < k$. We define $[k]!$ and $\begin{bmatrix}
k \\
l
\end{bmatrix}$ as follows.
\begin{align*}
[k]!&:= [k][k-1]\cdots [1] \\
\begin{bmatrix}
k \\
l
\end{bmatrix}
&:=\frac{[k]!}{[l]![k-l]!}
\end{align*}
\end{dfn}

\begin{lem} \label{2term}
We have
\begin{align*}
\begin{bmatrix}
k+1 \\
j
\end{bmatrix}
=q^{-k+j-1}\begin{bmatrix}
k \\
j-1
\end{bmatrix}
+q^j\begin{bmatrix}
k \\
j
\end{bmatrix}
\end{align*}
\end{lem}
\begin{proof}
It is given by directly computing.
\begin{align*}
q^{-k+j-1}\begin{bmatrix}
k \\
j-1
\end{bmatrix}
+q^j\begin{bmatrix}
k \\
j
\end{bmatrix}
&= \frac{[k]!}{[j]![k+1-j]!}(q^{-k+j-1}[j]+q^j[k+1-j]) \\
&= \frac{[k]!}{[j]![k+1-j]!(q-q^{-1})}(q^{-k+2j-1}-q^{-k-1}+q^{k+1}-q^{-k+2j-1}) \\
&= \begin{bmatrix}
k+1 \\
j
\end{bmatrix}
\end{align*}
\end{proof}

\begin{dfn}
Let $\mu$ be a complex number without non-negative integer. Verma module $M(\mu)$ over $\mathbb{C}(q,q^\mu)$ has a basis $\{v_0, v_1, \cdots\}$ called an induced basis and satisfies the following relations for the generators of $U_q(\mathfrak{sl}_2)$.
\begin{align*}
K\cdot v_i:= q^{\mu-2i}v_i \\
E\cdot v_i:= [i]v_{i-1} \\
F\cdot v_i:= [\mu-i]v_{i+1}
\end{align*}
where $v_{-1}=0$.
\end{dfn}

\begin{dfn}
Given a module $M, M'$ on $U_q(\mathfrak{sl}_2)$, by using the action induced from the coproduct, the tensor product representation $M \otimes M' $ is defined from the following relations. For all $m \in M $ and $m' \in M' $, we have
\begin{align*}
K^{\pm1}\cdot(m\otimes m'):=(K^{\pm1}\cdot m)\otimes (K^{\pm1}\cdot m') \\
F\cdot (m\otimes m'):=(F\cdot m)\otimes m'+(K^{-1}\cdot m)\otimes (F\cdot m') \\
E\cdot (m\otimes m'):=(E\cdot m)\otimes (K\cdot m')+m \otimes(E\cdot m')
\end{align*}
\end{dfn}

\section{Endomorphism algebras of $V_1^{\otimes n}$ and $M(\mu_1)\otimes\cdots \otimes M(\mu_n)$}
In this section, we introduce Temperley-Lieb algebra and Temperley-Lieb algebra of type B. Hereinafter, we simply denote $v_i \otimes v_j \in M \otimes N$ by $v_{i,j}$ where $M$ and $N$ are $M(\mu)$ or $V_k$ respectively.

\begin{dfn}
The intertwining operators $cap: \mathbb{C}(q) \rightarrow V_1\otimes V_1$ and $cup: V_1\otimes V_1 \rightarrow \mathbb{C}(q)$ over $U_q\mathfrak{sl}_2$ are defined as follows.
\begin{align*}
&cap(1)=v_{0,1}-q^{-1}v_{1,0} \\
&cup(v_{0,0})=cup(v_{1,1})=0, \ cup(v_{0,1})=-q, \ cup(v_{1,0})=1
\end{align*}
\end{dfn}
\begin{dfn}
Temperley Lieb algebra $TL_n$ is a $\mathbb{C}(q)$-algebra $\mathrm{End}_{U_q(\mathfrak{sl}_2)}(V_1^{\otimes n})$. The generators of $TL_n$ is $\{e_i\}$ ($i=1, \cdots, n-1$) where $e_i$ is defined as follows.
\[
e_i= \mathrm{Id}^{\otimes (i-1)}\otimes (cup \circ cap)\otimes \mathrm{Id}^{\otimes (n-i-1)}
\]
\end{dfn}

\begin{dfn}
We let $TL_{\mu_1,\cdots,\mu_n}:= End_{U_q(\mathfrak{sl}_2)}(M(\mu_1)\otimes\cdots \otimes M(\mu_n))$.
\end{dfn}

\begin{prop}
Let the ground field of $U_q(\mathfrak{sl}_2)$, $M(\mu)\otimes M(\lambda)$ and $M(\mu+\lambda)$ be $\mathbb{C}(q,q^\mu,q^\lambda)$. Set $\mathbb{C}(q, q^\mu, q^\lambda)$ linear maps $E_{\mu, \lambda}: M(\mu) \otimes M(\lambda) \rightarrow M(\mu+\lambda)$ and $F_{\mu, \lambda}: M(\mu+\lambda) \rightarrow M(\mu) \otimes M(\lambda)$ as follows.
\begin{align*}
E_{\mu, \lambda}(v_{i,j})&:= q^{i(\lambda-j)}v_{i+j} \\
F_{\mu, \lambda}(v_k)&:= \sum_{j=0}^k q^{-(k-j)(\mu-j)}\begin{bmatrix}
k \\
j
\end{bmatrix}
\frac{\prod_{i=0}^{j-1}[\mu-i] \prod_{i=0}^{k-j-1}[\lambda-i]}{\prod_{i=0}^{k-1}[\mu+\lambda-i]}v_{j,k-j}
\end{align*}
where $\prod_{i=0}^{-1}:=1$. Then we have $E_{\mu, \lambda} \in \mathrm{Hom}_{U_q(\mathfrak{sl}_2)}(M(\mu) \otimes M(\lambda), M(\mu+\lambda))$ and $F_{\mu,\lambda} \in \mathrm{Hom}_{U_q(\mathfrak{sl}_2)}(M(\mu+\lambda), M(\mu) \otimes M(\lambda))$.
\end{prop}
\begin{proof}
We must show the commutativity $XE_{\mu,\lambda}=E_{\mu, \lambda}X$ and $XF_{\mu, \lambda}=F_{\mu, \lambda}X$ where $X=K, E, F$. For $i,j \geq 0$, the action below is defined as follows.
\begin{align*}
K, E, F: M(\mu)\otimes M(\lambda) \rightarrow M(\mu)\otimes M(\lambda) \\
\end{align*}
\begin{align*}
Kv_{i,j}&= q^{\mu+\lambda-2(i+j)}v_{i,j} \\
Ev_{i,j}&= q^{\lambda-2j}[i]v_{i-1,j}+[j]v_{i,j-1} \\
Fv_{i,j}&= [\mu-i]v_{i+1,j}+q^{-\mu+2i}[\lambda-j]v_{i,j+1}
\end{align*}
First we will show that $E_{\mu}$ is an intertwining operator over $U_q(\mathfrak{sl}_2)$. Now we prove the following diagram is commutative.
\[
\xymatrix{
M(\mu)\otimes M(\lambda) \ar[r]^{E_{\mu,\lambda}} \ar[d]^{X} & M(\mu+\lambda) \ar[d]^{X} \\
M(\mu)\otimes M(\lambda) \ar[r]^{E_{\mu,\lambda}} & M(\mu+\lambda)
}
\]
Consider the case $X=K$, we obtain
\begin{align*}
&KE_{\mu,\lambda}v_{i,j}= q^{\mu+\lambda-2(i+j)}q^{i(\lambda-j)}v_{i+j} \\
&E_{\mu,\lambda}Kv_{i,j}= q^{i(\lambda-j)}q^{\mu+\lambda-2(i+j)}v_{i+j}
\end{align*}
Then we have $KE_{\mu,\lambda}=E_{\mu,\lambda}K$.
Consider the case $X=E$, we obtain
\begin{align*}
EE_{\mu,\lambda}v_{i,j}&= [i+j]q^{i(\lambda-j)}v_{i+j-1} \\
E_{\mu,\lambda}Ev_{i,j}&= q^{(i-1)(\lambda-j)}q^{\lambda-2j}[i]v_{i+j-1}+q^{i(\lambda-j+1)}[j]v_{i+j-1} \\
&= q^{i(\lambda-j)}(q^{-j}[i]+q^i[j])v_{i+j-1} \\
&= q^{i(\lambda-j)}[i+j]v_{i+j-1}
\end{align*}
Then we have $EE_{\mu,\lambda}=E_{\mu,\lambda}E$.
Consider the case $X=F$, we obtain
\begin{align*}
FE_{\mu,\lambda}v_{i,j}&= [\mu+\lambda-i-j]q^{i(\lambda-j)}v_{i+j+1} \\
E_{\mu,\lambda}Fv_{i,j}&= q^{(i+1)(\lambda-j)}[\mu-i]v_{i+j+1}+q^{i(\lambda-j-1)}q^{-\mu+2i}[\lambda-j]v_{i+j+1} \\
&= q^{i(\lambda-j)}(q^{\lambda-j}[\mu-i]+q^{-\mu+i}[\lambda-j])v_{i+j+1} \\
&= q^{i(\lambda-j)}[\mu+\lambda-i-j]v_{i+j+1}
\end{align*}
Then we have $FE_{\mu,\lambda}=E_{\mu,\lambda}F$. From the computations above, $E_{\mu,\lambda}$ is an intertwining operator over $U_q(\mathfrak{sl}_2)$. Next we will show that $F_{\mu,\lambda}$ is an intertwining operator over $U_q(\mathfrak{sl}_2)$. Now we prove the following diagram is commutative.
\[
\xymatrix{
M(\mu+\lambda) \ar[r]^{F_{\mu,\lambda}} \ar[d]^{X} & M(\mu)\otimes M(\lambda) \ar[d]^{X} \\
M(\mu+\lambda) \ar[r]^{F_{\mu,\lambda}} & M(\mu)\otimes M(\lambda)
}
\]
Consider the case $X=K$, we obtain
\begin{align*}
KF_{\mu,\lambda}v_k&= q^{\mu+\lambda-2k}\sum_{j=0}^k q^{-(k-j)(\mu-j)}\begin{bmatrix}
k \\
j
\end{bmatrix}
\frac{\prod_{i=0}^{j-1}[\mu-i] \prod_{i=0}^{k-j-1}[\lambda-i]}{\prod_{i=0}^{k-1}[\mu+\lambda-i]}v_{j,k-j} \\
F_{\mu,\lambda}Kv_k&= \sum_{j=0}^k q^{-(k-j)(\mu-j)}\begin{bmatrix}
k \\
j
\end{bmatrix}
\frac{\prod_{i=0}^{j-1}[\mu-i] \prod_{i=0}^{k-j-1}[\lambda-i]}{\prod_{i=0}^{k-1}[\mu+\lambda-i]}q^{\mu+\lambda-2k}v_{j,k-j}
\end{align*}
Then we have $KF_{\mu,\lambda}=F_{\mu,\lambda}K$. Consider the case $X=E$. If $k=0$, $EF_{\mu,\lambda}v_k= F_{\mu,\lambda}Ev_k$ is trivial. If $k \neq 0$, we obtain
\begin{align*}
&EF_{\mu,\lambda}v_k= \\
&\sum_{j=1}^k q^{-(k-j)(\mu-j)}\begin{bmatrix}
k \\
j
\end{bmatrix}
\frac{\prod_{i=0}^{j-1}[\mu-i] \prod_{i=0}^{k-j-1}[\lambda-i]}{\prod_{i=0}^{k-1}[\mu+\lambda-i]}q^{\lambda-2(k-j)}[j]v_{j-1,k-j} \\
& \qquad +\sum_{i=0}^{k-1}q^{-(k-j)(\mu-j)}\begin{bmatrix}
k \\
j
\end{bmatrix}
\frac{\prod_{i=0}^{j-1}[\mu-i] \prod_{i=0}^{k-j-1}[\lambda-i]}{\prod_{i=0}^{k-1}[\mu+\lambda-i]}[k-j]v_{j,k-1-j} \\
&= \sum_{j=0}^{k-1} q^{-(k-1-j)(\mu-j-1)}\begin{bmatrix}
k \\
j+1
\end{bmatrix}
\frac{\prod_{i=0}^j[\mu-i] \prod_{i=0}^{k-j-2}[\lambda-i]}{\prod_{i=0}^{k-1}[\mu+\lambda-i]}q^{\lambda-2(k-1-j)}[j+1]v_{j,k-1-j} \\
& \qquad +\sum_{i=0}^{k-1}q^{-(k-1-j)(\mu-j)}\begin{bmatrix}
k-1 \\
j
\end{bmatrix}
\frac{\prod_{j=0}^{j-1}[\mu-i] \prod_{i=0}^{k-j-2}[\lambda-i]}{\prod_{i=0}^{k-2}[\mu+\lambda-i]}\frac{q^{-\mu+j}[k][\lambda-k+j+1]}{[\mu+\lambda-k+1]}v_{j,k-1-j} \\
&= \sum_{j=0}^{k-1} q^{-(k-1-j)(\mu-j)}\begin{bmatrix}
k-1 \\
j
\end{bmatrix}
\frac{\prod_{i=0}^{j-1}[\mu-i] \prod_{i=0}^{k-j-2}[\lambda-i]}{\prod_{i=0}^{k-2}[\mu+\lambda-i]}\frac{q^{\lambda-k+j+1}[k][\mu-j]}{[\mu+\lambda-k+1]}v_{j,k-1-j} \\
& \qquad +\sum_{j=0}^{k-1}q^{-(k-1-j)(\mu-j)}\begin{bmatrix}
k-1 \\
j
\end{bmatrix}
\frac{\prod_{i=0}^{j-1}[\mu-i] \prod_{i=0}^{k-j-2}[\lambda-i]}{\prod_{i=0}^{k-2}[\mu+\lambda-i]}\frac{q^{-\mu+j}[k][\lambda-k+j+1]}{[\mu+\lambda-k+1]}v_{j,k-1-j} \\
&= [k]\sum_{j=0}^{k-1} q^{-(k-1-j)(\mu-j)}\begin{bmatrix}
k-1 \\
j
\end{bmatrix}
\frac{\prod_{i=0}^{j-1}[\mu-i] \prod_{i=0}^{k-j-2}[\lambda-i]}{\prod_{i=0}^{k-2}[\mu+\lambda-i]} \\
& \qquad\qquad \cdot\frac{q^{\lambda-k+j+1}[\mu-j]+q^{-\mu+j}[\lambda-k+j+1]}{[\mu+\lambda-k+1]}v_{j,k-1-j} \\
&= [k]\sum_{j=0}^{k-1} q^{-(k-1-j)(\mu-j)}\begin{bmatrix}
k-1 \\
j
\end{bmatrix}
\frac{\prod_{i=0}^{j-1}[\mu-i] \prod_{i=0}^{k-j-2}[\lambda-i]}{\prod_{i=0}^{k-2}[\mu+\lambda-i]}v_{j,k-1-j}
\end{align*}
\begin{align*}
F_{\mu,\lambda}Ev_{i,j}= [k]\sum_{j=0}^{k-1} q^{-(k-1-j)(\mu-j)}\begin{bmatrix}
k-1 \\
j
\end{bmatrix}
\frac{\prod_{i=0}^{j-1}[\mu-i] \prod_{i=0}^{k-j-2}[\lambda-i]}{\prod_{i=0}^{k-2}[\mu+\lambda-i]}v_{j,k-1-j}
\end{align*}
Then we have $EF_{\mu,\lambda}=F_{\mu,\lambda}E$. Consider the case $X=F$, we obtain
\begin{align*}
FF_{\mu,\lambda}v_k&= \sum_{j=0}^k q^{-(k-j)(\mu-j)}\begin{bmatrix}
k \\
j
\end{bmatrix}
\frac{\prod_{i=0}^{j-1}[\mu-i]\prod_{i=0}^{k-j-1}[\lambda-i]}{\prod_{i=0}^{k-1}[\mu+\lambda-i]}Fv_{j,k-j} \\
&= \sum_{j=0}^k q^{-(k-j)(\mu-j)}\begin{bmatrix}
k \\
j
\end{bmatrix}
\frac{\prod_{i=0}^{j-1}[\mu-i]\prod_{i=0}^{k-j-1}[\lambda-i]}{\prod_{i=0}^{k-1}[\mu+\lambda-i]}[\mu-j]v_{j+1,k-j} \\
&\qquad +\sum_{j=0}^k q^{-(k-j)(\mu-j)}\begin{bmatrix}
k \\
j
\end{bmatrix}
\frac{\prod_{i=0}^{j-1}[\mu-i]\prod_{i=0}^{k-j-1}[\lambda-i]}{\prod_{i=0}^{k-1}[\mu+\lambda-i]}q^{-\mu+2j}[\lambda-k+j]v_{j,k+1-j} \\
&= \sum_{j=1}^{k+1} q^{-(k+1-j)(\mu-j+1)}\begin{bmatrix}
k \\
j-1
\end{bmatrix}
\frac{\prod_{i=0}^{j-2}[\mu-i]\prod_{i=0}^{k-j}[\lambda-i]}{\prod_{i=0}^{k-1}[\mu+\lambda-i]}[\mu-j+1]v_{j,k+1-j} \\
&\qquad +\sum_{j=0}^k q^{-(k+1-j)(\mu-j)}\begin{bmatrix}
k \\
j
\end{bmatrix}
\frac{\prod_{i=0}^{j-1}[\mu-i]\prod_{i=0}^{k-j}[\lambda-i]}{\prod_{i=0}^{k-1}[\mu+\lambda-i]}q^j v_{j,k+1-j} \\
&= \sum_{j=0}^{k+1} q^{-(k+1-j)(\mu-j)}\biggl(q^{-k-1+j}\begin{bmatrix}
k \\
j-1
\end{bmatrix}
+q^j\begin{bmatrix}
k \\
j
\end{bmatrix}
\biggr)\frac{\prod_{i=0}^{j-1}[\mu-i]\prod_{i=0}^{k-j}[\lambda-i]}{\prod_{i=0}^{k-1}[\mu+\lambda-i]}v_{j,k+1-j} \\
&= \sum_{j=0}^{k+1} q^{-(k+1-j)(\mu-j)}\begin{bmatrix}
k+1 \\
j
\end{bmatrix}
\frac{\prod_{i=0}^{j-1}[\mu-i]\prod_{i=0}^{k-j}[\lambda-i]}{\prod_{i=0}^{k-1}[\mu+\lambda-i]}v_{j,k+1-j} \\
&\qquad (\because \mathrm{Lemma} \ \ref{2term})
\end{align*}
\begin{align*}
F_{\mu,\lambda}Fv_k&= [\mu+\lambda-k]F_{\mu,\lambda}v_{k+1} \\
&= [\mu+\lambda-k]\sum_{j=0}^{k+1} q^{-(k+1-j)(\mu-j)}\begin{bmatrix}
k+1 \\
j
\end{bmatrix}
\frac{\prod_{i=0}^{j-1}[\mu-i]\prod_{i=0}^{k-j}[\lambda-i]}{\prod_{i=0}^k[\mu+\lambda-i]}v_{j,k+1-j} \\
&= \sum_{j=0}^{k+1} q^{-(k+1-j)(\mu-j)}\begin{bmatrix}
k+1 \\
j
\end{bmatrix}
\frac{\prod_{i=0}^{j-1}[\mu-i]\prod_{i=0}^{k-j}[\lambda-i]}{\prod_{i=0}^{k-1}[\mu+\lambda-i]}v_{j,k+1-j}
\end{align*}
Then we have $FF_{\mu,\lambda}=F_{\mu,\lambda}F$. From the results above, $F_{\mu,\lambda}$ is an intertwining operator.
\end{proof}

\section{Main Theorem}
In this section, we define special idempotents in $TL_{\mu_1,\cdots,\mu_n}$ like the Jones Wenzl projector in $TL_n$. Hereinafter, let the ground field of $U_q(\mathfrak{sl}_2)$ and $M(\mu_1)\otimes \cdots \otimes M(\mu_n)$ be $\mathbb{C}(q,q^{\mu_1},\cdots,q^{\mu_n})$.
\begin{dfn}
Jones Wenzl projector $P_n \in TL_n$ is defined as follows.
\begin{align*}
P_1:=\mathrm{Id}, \quad P_n:=P_{n-1}+\frac{[n-1]}{[n]}P_{n-1}e_{n-1}P_{n-1}
\end{align*}
\end{dfn}

\begin{prop}
We have
\begin{align*}
P_n^2&=P_n \\
\end{align*}
\end{prop}
\begin{proof}
See Proposition 2 in \cites{KauffmanLins}.
\end{proof}

\begin{dfn}
Put $E_{\mu,\lambda;\mu_1,\cdots,\mu_i}:= E_{\mu,\lambda} \otimes \mathrm{Id}_{\mu_1}\otimes\cdots \otimes \mathrm{Id}_{\mu_i}$ and $F_{\mu,\lambda;\mu_1,\cdots,\mu_i}:= F_{\mu,\lambda} \otimes \mathrm{Id}_{\mu_1}\otimes\cdots \otimes \mathrm{Id}_{\mu_i}$. By induction on $n>2$, extended Jones Wenzl projectors $P_{\mu_1,\cdots,\mu_n} \in TL_{\mu_1,\cdots,\mu_n}$ are defined as follows.
\begin{align*}
P_{\mu_1,\mu_2}:= F_{\mu_1,\mu_2}E_{\mu_1,\mu_2}, \quad P_{\mu_1,\cdots,\mu_n}:=F_{\mu_1,\mu_2;\mu_3,\cdots,\mu_n}P_{\mu_1+\mu_2,\mu_3,\cdots,\mu_n}E_{\mu_1,\mu_2;\mu_3,\cdots,\mu_n}
\end{align*}
\end{dfn}
\begin{rem}
The definition above is inspired by Definition 2.11 in \cites{RoseTubbenhauer}. By Corollary 2.13 in \cites{RoseTubbenhauer}, the Jones Wenzl projectors defined in \cites{RoseTubbenhauer} coincide with $P_n$.
\end{rem}

\begin{lem} \label{EF=[mu+1]}
We have
\[
E_{\mu,\lambda}F_{\mu,\lambda}= \mathrm{Id}_{\mu+\lambda}
\]
\end{lem}
\begin{proof}
Let $k$ be a non-negative integer. Then we have
\begin{align*}
E_{\mu,\lambda}F_{\mu,\lambda}v_k&= \sum_{j=0}^k q^{-(k-j)(\mu-j)}\begin{bmatrix}
k \\
j
\end{bmatrix}
\frac{\prod_{i=0}^{j-1}[\mu-i] \prod_{i=0}^{k-j-1}[\lambda-i]}{\prod_{i=0}^{k-1}[\mu+\lambda-i]}q^{j(\lambda-k+j)}v_k \\
&= \sum_{j=0}^k q^{j(\mu+\lambda)-k\mu}\begin{bmatrix}
k \\
j
\end{bmatrix}
\frac{\prod_{i=0}^{j-1}[\mu-i] \prod_{i=0}^{k-j-1}[\lambda-i]}{\prod_{i=0}^{k-1}[\mu+\lambda-i]}v_k
\end{align*}
Now we prove $E_{\mu,\lambda}F_{\mu,\lambda}v_k= v_k$ by induction on $k$. If $k=0$, it is trivial by the above. By the above calculation, it suffices to show that
\begin{align*}
\prod_{i=0}^k[\mu+\lambda-i]=\sum_{j=0}^{k+1}q^{j(\mu+\lambda)-(k+1)\mu}\begin{bmatrix}
k+1 \\
j
\end{bmatrix}
\prod_{i=0}^{j-1}[\mu-i]\prod_{i=0}^{k-j}[\lambda-i]
\end{align*}
By induction, we obtain
\begin{align*}
\prod_{i=0}^k[\mu+\lambda-i]&= [\mu+\lambda-k]\prod_{i=0}^{k-1}[\mu+\lambda-i] \\
&= [\mu+\lambda-k]\sum_{j=0}^k q^{j(\mu+\lambda)-k\mu}\begin{bmatrix}
k \\
j
\end{bmatrix}
\prod_{i=0}^{j-1}[\mu-i]\prod_{i=0}^{k-j-1}[\lambda-i] \\
&\qquad (\because \mathrm{by} \ \mathrm{induction}) \\
&= \sum_{j=0}^k q^{j(\mu+\lambda)-k\mu}\begin{bmatrix}
k \\
j
\end{bmatrix}
q^{\lambda-k+j}\prod_{i=0}^j[\mu-i]\prod_{i=0}^{k-j-1}[\lambda-i] \\
&\qquad +\sum_{j=0}^k q^{j(\mu+\lambda)-k\mu}\begin{bmatrix}
k \\
j
\end{bmatrix}
q^{-\mu+j}\prod_{i=0}^{j-1}[\mu-i]\prod_{i=0}^{k-j}[\lambda-i] \\
&\qquad (\because [\mu+\lambda-k]=q^{\lambda-k+j}[\mu-j]+q^{-\mu+j}[\lambda-k+j]) \\
&= \sum_{j=1}^{k+1}q^{j(\mu+\lambda)-(k+1)\mu}\prod_{i=0}^{j-1}[\mu-i]\prod_{i=0}^{k-j}[\lambda-i]q^{-k+j-1}\begin{bmatrix}
k \\
j-1
\end{bmatrix} \\
&\qquad +\sum_{j=0}^k q^{j(\mu+\lambda)-(k+1)\mu}\prod_{i=0}^{j-1}[\mu-i]\prod_{i=0}^{k-j}[\lambda-i]q^j\begin{bmatrix}
k \\
j
\end{bmatrix} \\
&= \sum_{j=1}^{k+1}q^{j(\mu+\lambda)-(k+1)\mu}\begin{bmatrix}
k+1 \\
j
\end{bmatrix}
\prod_{i=0}^{j-1}[\mu-i]\prod_{i=0}^{k-j}[\lambda-i] \\
&\qquad (\because \mathrm{Lemma} \ \ref{2term})
\end{align*}
Then the result follows.
\end{proof}

\begin{thm} \label{theorem}
Extended Jones Wenzl projectors $P_{\mu_1,\ldots,\mu_n} \in TL_{\mu_1,\ldots,\mu_n}$ satisfy the following conditions.
\begin{align*}
P_{\mu_1,\ldots,\mu_n}^2= P_{\mu_1,\ldots,\mu_n}
\end{align*}
\end{thm}
\begin{proof}
We prove it by induction on $n>2$. If $n=2$, from Lemma \ref{EF=[mu+1]}, we have
\begin{align*}
P_{\mu_1,\mu_2}^2&= F_{\mu_1,\mu_2}E_{\mu_1,\mu_2}F_{\mu_1,\mu_2}E_{\mu_1,\mu_2} \\
&= F_{\mu_1,\mu_2}E_{\mu_1,\mu_2}
\end{align*}
Then we obtain $P_{\mu_1,\mu_2}^2=P_{\mu_1,\mu_2}$. Suppose that $P_{\mu_1,\ldots,\mu_{n-1}}^2= P_{\mu_1,\ldots,\mu_{n-1}}$. Then we have
\begin{align*}
P_{\mu_1,\ldots,\mu_n}^2&= F_{\mu_1,\mu_2;\mu_3,\cdots,\mu_n}P_{\mu_1+\mu_2,\mu_3,\cdots,\mu_n}E_{\mu_1,\mu_2;\mu_3,\cdots,\mu_n}F_{\mu_1,\mu_2;\mu_3,\cdots,\mu_n}P_{\mu_1+\mu_2,\mu_3,\cdots,\mu_n}E_{\mu_1,\mu_2;\mu_3,\cdots,\mu_n} \\
&= F_{\mu_1,\mu_2;\mu_3,\cdots,\mu_n}P_{\mu_1+\mu_2,\mu_3,\cdots,\mu_n}^2 E_{\mu_1,\mu_2;\mu_3,\cdots,\mu_n} \quad (\because \mathrm{Lemma} \ \ref{EF=[mu+1]}) \\
&= F_{\mu_1,\mu_2;\mu_3,\cdots,\mu_n}P_{\mu_1+\mu_2,\mu_3,\cdots,\mu_n}E_{\mu_1,\mu_2;\mu_3,\cdots,\mu_n} \quad (\because \mathrm{By} \ \mathrm{induction}) \\
&= P_{\mu_1,\cdots,\mu_n}
\end{align*}
Thus the result follows.
\end{proof}

\footnotesize{DEPARTMENT OF MATHEMATICS, TOHOKU UNIVERSITY, 6-3, AOBA, ARAMAKI-AZA, AOBA-KU, SENDAI, 980-8578, JAPAN} \par
\footnotesize{\textit{Email Address}}: \texttt{matsumoto.ryoga.t2@dc.tohoku.ac.jp}

\end{document}